\documentclass[11pt]{article}
\usepackage[lined,ruled]{algorithm2e}

\usepackage{graphicx}   
\usepackage{multirow}
\usepackage{amsmath}
\usepackage{amsfonts}
\usepackage{latexsym}
\usepackage{verbatim}
\usepackage{amsmath}    

\newtheorem{theorem}{Theorem}
\newtheorem{lemma}[theorem]{Lemma}
\newtheorem{corollary}[theorem]{Corollary}

\newtheorem{remark}[theorem]{Remark}

\newtheorem{assumption}{A}

\newcommand{\beq}{\begin{equation}}
\newcommand{\eeq}{\end{equation}}
\newcommand{\beqa}{\begin{eqnarray}}
\newcommand{\eeqa}{\end{eqnarray}}
\newcommand{\beqas}{\begin{eqnarray*}}
\newcommand{\eeqas}{\end{eqnarray*}}
\newcommand{\bi}{\begin{itemize}}
\newcommand{\ei}{\end{itemize}}
\newcommand{\ba}{\begin{array}}
\newcommand{\ea}{\end{array}}

\newcommand{\bbr}{\Bbb{R}}

\title{
Linearly Convergent First-Order Algorithms\\ for Semi-definite Programming
}
\author{
   Cong D. Dang, Guanghui Lan
    \thanks{Department of Industrial and Systems Engineering,
    University of Florida, Gainesville, FL 32611,
       (email: {\tt glan@ise.ufl.edu}). }
}

\begin{document}

\maketitle

\begin{abstract}
In this paper, we consider two formulations for Linear Matrix Inequalities (LMIs) under Slater type constraint qualification assumption, namely, SDP smooth and non-smooth formulations. We also propose two first-order linearly convergent algorithms for solving these formulations. Moreover, we introduce a bundle-level method which converges linearly uniformly for both smooth and non-smooth problems and does not require any smoothness information. The convergence properties of these algorithms are also discussed. Finally, we consider a special case of LMIs, linear system of inequalities, and show that a linearly convergent algorithm can be obtained under a weaker assumption. 
\vspace{.1in}


\end{abstract}

\vspace{0.1cm}
\setcounter{equation}{0}
\section{Introduction} \label{sec_intro}
Semi-definite Programming (SDP) is one of most interesting branches of mathematical programming in last twenty years. Semi-definite Programming can be used to model many practical problems in vary fields such as Convex constrained Optimization, Combinatorial Optimization, Control Theory,... We refer to \cite{VanBoy94} for a general survey and applications of SDP. Algorithms for solving SDP have been explosively studied since a major works are made by Nesterov and Nemirovski \cite{nene88_te1}, \cite{nene90_te1}, \cite{nene90_te2}, \cite{nene91_te1}, in which they showed that Interior Point (IP) methods for Linear Programming (LP) can be extended to SDP. Related topics can be found in \cite{StZh01}, \cite{LuStZh98}. Despite the fact that SDP can be solved in polynomial time by IP methods, they become impractical when the number of constraints increase because of computational cost per each iteration. Recently, first-order methods are focused because of the efficiency in solving large scale SDP such as Nesterov's optimal methods \cite{Nest05-1}, \cite{Nest06-1}, Nemirovski's prox-method \cite{Nem05-1} and spectral bundle methods \cite{HeRe99}.    

In system and control theory, system identification and signal processing, Semi-definite Programmings are used in context of Linear Matrix Inequalities constraints (LMIs), see \cite{BoElFeBa94}, \cite{BaVa98}. LMIs can also be solved numerically by recent interior point methods for semi-definite programming, see \cite{HeReVaWo96}, \cite{NeNe94}.  

Linear Programming is a special case of Semidefinite Programming, as well as Linear system of inequalities is a special case of Linear Mtrix Inequalities. Hence, any algorithms for SDP can be applied for solving LP. In this paper, we propose a linearly convergent algorithm for Linear system of inequalities, which require a weaker assumption than the one for LMIs problem. We refer to \cite{LeLe08} for other linear convergent algorithms for Linear system of inequalities.    

Error bounds usually play an important role in algorithmic convergence proofs. In particular, Luo and Tseng showed the power of error bound idea in deriving the linear convergent rate in many algorithm for variety class of problems, see \cite{LuTs92}, \cite{LuTs93a}, \cite{LuTs95}. However, it is not easy to obtain an error bound except in linear and quadratic cases, or when the Slater constraint qualification condition holds, see \cite{DeHu99}. In \cite{Zhang99}, Zhang derived error bounds for general convex conic problem under some various conditions. The error bound for Semidefinite Programming was studied by Deng and Hu in\cite{DeHu99}, Jourani and Ye in\cite{JoYe05}. Related topics can be found in \cite{Pang97}, \cite{StZh01}, \cite{LuStZh98}.

The paper is organized as follows. In Section 2, we introduce the problem of interest and the Slater constraint constraint qualification condition is made. Respectively, in Section 3 and Section 4, we present a non-smooth SDP optimization and a smooth SDP formulation and propose two different linearly convergent first order algorithms for solving these formulations. The iteration complexity for these algorithms are also derived. An uniformly linearly convergent algorithm for both formulations and its convergence properties are presented in Section 5. We also discuss about a special cases of LMIs, the linear system of inequalities in Section 6. Finally, we have some conclusions and remarks in the last section. 

\section{The problem of interest}
In this section, we first discuss about the relationship between a primal-dual SDP problem and a LMI. In particular, any primal-dual SPD problem can be represented by a LMI problem. 

Given a given linear operator ${\mathcal A}: \bbr^n \rightarrow {\mathcal S}^n$, vectors $c \in \bbr^n$ and matrix $b \in {\mathcal S}^n$, we consider the SDP problem
\beq \label{p}
\mathop {\min }\limits_x \left\{ {\left\langle {c,x} \right\rangle : {\mathcal A}x \preceq  B} \right\}
\eeq
and its associated dual problem
\beq \label{d}
\mathop {\max }\limits_y \left\{ {\left\langle {b,y} \right\rangle :{\mathcal A}^T y = c,y \preceq  0} \right\}
\eeq
 where $y \in {\mathcal S}^n.$ We make the following assumption.
 \begin{assumption} \label{feasiblePD}
 Both primal and dual SDP problems \eqref{p} and \eqref{d} are strictly feasible.
 \end{assumption}
 It is well known that in view of Assumption~\ref{feasiblePD}, the pair of primal and dual SDP problem \eqref{p} and \eqref{d} satisfy the Slater's condition, hence they have optimal solutions and their associated gap duality is zero, see \cite{BenNem00}. Moreover, a primal-dual optimal solution of \eqref{p} and \eqref{d} can be found by solving the complementarity problem as following Linear Matrix Inequalities constraints
\[
\left\{ \begin{array}{l}
 {\mathcal A}x \le B \\ 
 {\mathcal A}^T y = c \\ 
 y \le 0 \\ 
 \left\langle {x,c} \right\rangle  - \left\langle {B,y} \right\rangle  \le 0. \\ 
 \end{array} \right.
\]
Note that a system of Linear Matrix Inequalities (LMIs) is equivalent to a single LMI because of the simple fact that a system of LMIs can be easily represented by a single LMI, see \cite{BenNem00}. For convenience, from now on we just consider a single LMI problem.

Given a symmetric metrix $B$ and a linear operator $\mathcal{A} : \bbr^n \rightarrow \mathcal{S}^n$ as follow 
$$\mathcal{A}x=A_1x_1+...+A_nx_n,$$ 
where $\mathcal{S}^n$ denotes the set of $n \times n$ symmetric matrices and $A_1,A_2,...,A_n \in \mathcal{S}^n$, the problem of interest in this paper is finding a feasible solution $x \in \bbr^n$ to the %
Linear Matrix Inequality, assume that the feasible solution set $S$ is nonempty, 
\beq \label{SDP_inq}
\mathcal{A}x-B \preceq 0. 
\eeq
The Linear Matrix Inequality \eqref{SDP_inq} can be represented in the conic form
\beq \label{SDP_equivalent}
\mathcal{\tilde A}-B \in {\mathcal S}^n_-,
\eeq
where $\mathcal{\tilde A}$ is the span of $\{A_1,A_2,...,A_n\}.$ The following assumption is made throughout the paper
\begin{assumption} \label{assumption1} There exist $\sigma>0$ and $d \in \bbr^n$ such that
$$\sigma I_n - \mathcal{A}d \in \mathcal{S}^n_-,$$
and denote 
\beq \label{mu}
\mu=\frac{\| d\|}{\sigma}.
\eeq
\end{assumption}
Note that the Assumption~\ref{assumption1} implies the Slater constraint qualification condition for the feasible set of \eqref{SDP_equivalent}, hence $S$ is nonempty, see \cite{JoYe05}, \cite{Zhang99}, \cite{DeHu99}, \cite{De97}. In Section 2 and Section 3, we will present two equivalent SDP optimization formulations of LMI and linearly convergent algorithms for solving these formulations.
\section{A non-smooth SDP Optimization Formulation for LMI}
In this section, we introduce a non-smooth SDP Optimization formulation for the Linear Matrix Inequality \eqref{SDP_inq}. We also propose a linearly convergence algorithm for solving the non-smooth formulation and present the main convergence behavior of this algorithm.\\

Consider the alternative optimization problem that minimizing over $\bbr^n$ the objective function
\beq \label{SDP_nonsmooth}
f(x)=\max \{ \lambda_1(\mathcal{A}x-B),0,\}
\eeq
where $\lambda_1(\mathcal{A}x-B)$ denotes the maximum eigenvalue of $\mathcal{A}x-B.$ Clearly, the objective function is not differentiable and the problem \eqref{SDP_nonsmooth} is non-smooth. Note the, computing the value and the subgradient of the objective function requires to find a maximal eigenvalue and its associated eigenvectors. 
 The objective function $f(x)$ is Lipschitz continuous, i.e. for any $g \in \partial f(x),$ there exists a positive number $M$ such that
$$\| g(x)\| \le M \; \forall x \in \bbr^n.$$ The constant $M$ can be computed as follows
\beq \label{M_def}
M=\| \mathcal{A}\|=\sqrt{\sum_{i=1}^n \| A_i\|^2}, 
\eeq
where $\| A_i\|$ is operator norm (spectral norm or F-norm).\\

Furthermore, the two problem \eqref{SDP_nonsmooth} and \eqref{SDP_inq} are equivalent in the following sense. It is not difficult to see that, if $x^*$ is an optimal solution to \eqref{SDP_nonsmooth}, then $x^*$ is also a feasible solution to \eqref{SDP_inq} and vice versa. In addition, the optima value of \eqref{SDP_nonsmooth} is $F^*=0$. \\

Denote $X^*$ by the optimal solution set of \eqref{SDP_nonsmooth}. The following technical lemma describes the relation between the distance from an arbitrary point $x$ to the the optimal set $X^*$ and the objective function value at that point.
\begin{lemma}
For any $x \in \bbr^n$, we have
\beq \label{SDP_growth2}
d(x,X^*) \le \mu f(x), 
\eeq
where $X^*$ is the feasible solution set of \eqref{SDP_inq}.
\end{lemma}
\begin{proof}
Note that $X^*$ is also the optimal set of minimizing \eqref{SDP_nonsmooth}. We consider the following two cases.

\noindent{\bf Case 1:} $x \in X^*.$ Obviously, $d(x,X^*)=0$ and
$$f(x)=\min \{ \lambda_1(\mathcal{A}x-B),0\}=0.$$
That implies \eqref{SDP_growth2} is true for any $x \in X^*.$

\noindent{\bf Case 2:} $x \notin X^*.$ Clearly, the result is implied by Corollary 1 in \cite{JoYe05}.

The Lemma immediately follows from two cases.
\end{proof}
The relation \eqref{SDP_growth2} is also called the growth condition of the objective function. We now ready to describe our non-smooth algorithm as follows. Each main Step (Step 1), to obtain the new iterate, we run the sub-gradient method (see \cite{Nest04}) for $K=4M^2\mu^2,$ where $\mu$ is defined in \eqref{mu}, with the input is the current iterate. In other words, we restart the sub-gradient algorithm after a constant number $K=4M^2\mu^2$ of iterations. We denote $\{x_k\}, k=0,1,...$ by the sequence obtained by our algorithm and $\{\bar x_i\}, i=0,1,...$ by the sequence obtained by sub-gradient method in the Step 1. The non-smooth algorithm scheme is described as follows.\\









{\bf The SDP Non-Smooth Algorithm:}

{\bf Input:} $x_0^u \in \bbr^n$.

{\bf Output:} $x_k \in \bbr^n.$

1) $k^{th}$ iteration, $k \ge 1.$

\qquad \qquad Run sub-gradient algorithm with initial solution $\bar x_0=x_{k-1}$ for $K=4M^2\mu^2$ iterations.

\qquad \qquad $f^*_k:=\min_{i=1,...,K} f(\bar x_i)$.

\qquad \qquad $x_k:=\bar x$ such that $f(\bar x) =f^*_k $.

2) Go to Step 1.\\

Similar to the smooth algorithm, the non-smooth algorithm is definitely different from running sub-gradient method for multiple times because of the restarting of parameters. In order to prove the main convergence result of our algorithm, we show in the following lemma that the sub-gradient method applied in Step 1 has $O(\frac{1}{\sqrt{K}}) $ rate of convergence.
\begin{lemma} \label{subgrad_convergence}
Suppose that $\{\bar x_i, i=1,2,..,K\}$ is generated by sub-gradient method in each main Step. Then we have
$$\min_{i=1,2,...,K}f(\bar x_i)-f^* \le \frac{M d(\bar x_0,X^*)}{\sqrt{K}},$$
where $X^*$ is optimal solution set of \eqref{SDP_nonsmooth} and $M$ is defined in \eqref{M_def}.
\end{lemma}
\begin{proof}
For any $i \ge 1 $ and $x^* \in X^*,$ we have
$$\frac{1}{2}\|\bar x_{i+1}-x^* \|^2=\frac{1}{2}\| \bar x_i-x*\|^2-\gamma_i \langle g(\bar x_i),\bar x_i-x^* \rangle + \frac{\gamma_i^2}{2} \| g(\bar x_i)\|^2$$
or
$$\gamma_i \langle g(\bar x_i),\bar x_i-x^* \rangle =\frac{1}{2}\| \bar x_i-x*\|^2-\frac{1}{2}\|\bar x_{i+1}-x^* \|^2+ \frac{\gamma_i^2}{2} \| g(\bar x_i)\|^2.$$
Because the objective function is convex, then
$$f(\bar x_i)-f(x^*) \le \langle g(\bar x_i),\bar x_i-x^* \rangle,$$
that implies
$$\gamma_i [f(\bar x_i)-f(x^*)] \le \frac{1}{2}\| \bar x_i-x*\|^2-\frac{1}{2}\|\bar x_{i+1}-x^* \|^2+ \frac{\gamma_i^2}{2} \| g(\bar secx_i)\|^2.$$
Summing up the above inequalities we obtain
\begin{align*}
\sum_{i=1}^K \gamma_i [f(\bar x_i)-f(x^*)] &\le \frac{1}{2}\| \bar x_0-x*\|^2 -\frac{1}{2}\| \bar x_i-x*\|^2 + \frac{1}{2}\sum_{i=1}^K \gamma_i^2 \| g(\bar x_i)\|^2 \cr
&\le \frac{1}{2}\| \bar x_0-x*\|^2 + \frac{1}{2}\sum_{i=1}^K \gamma_i^2 \| g(\bar x_i)\|^2
\end{align*}
Dividing both sides to $\sum_{i=1}^K \gamma_i,$ and using Lemma~\ref{SDP_growth2}, we have
\begin{align*}
\frac{\sum_{i=1}^K \gamma_i [f(\bar x_i)-f^*]}{\sum_{i=1}^K \gamma_i} &\le \frac{\|\bar x_0-x^*\|^2}{2\sum_{i=1}^K \gamma_i}+\frac{\sum_{i=1}^K \gamma_i^2 M^2}{2\sum_{i=1}^K \gamma_k} \cr
&\le \frac{\mu^2 f^2(\bar x_0)}{2\sum_{i=1}^K \gamma_i}+\frac{\sum_{i=1}^K \gamma_i^2 M^2}{2\sum_{i=1}^K \gamma_i}
\end{align*}
We consider the constant step size $\gamma_i=\frac{\gamma}{\sqrt{K}},$ then the above relation becomes
$$\min_{i=1,2,..,K}\{f(\bar x_i)-f^*\} \le \frac{1}{2\sqrt{K}}[\frac{\mu^2 f^2(\bar x_0)}{\gamma}+M^2 \gamma].$$
Minimizing the right hand side, we find that the optimal choice is $\gamma=\frac{\mu f(\bar x_0)}{M}.$ In this case, we obtain the following rate of convergence:
$$f^*_k-f^* \le \frac{M \mu f(\bar x_0)}{\sqrt{K}}=\frac{M \mu f(x_{k-1})}{\sqrt{K}}.$$
That implies the $O(\frac{1}{\sqrt{K}})$ rate of convergence of sub-gradient method in each main Step of our algorithm.
\end{proof}
The linear convergence of our algorithm is stated in the following theorem.
\begin{theorem} \label{nonsmooth_convergence}
The sequence $\{x_k\}, k=0,1,...$ generated by the SDP smooth algorithm satisfies
$$f^*_k - f^* \le \frac{1}{2}[f^*_{k-1}-f^*], \forall k=1,2,...$$
\end{theorem}
\begin{proof}
By convergence properties of sub-gradient algorithm, we have
$$f^*_k-f^* \le \frac{M \mu f^*_{k-1}}{\sqrt{K}}.$$
Note that $f^*=0$ and $K \ge 4M^2\mu^2$, that implies
$$f^*_k - f^* \le \frac{1}{2}[f^*_{k-1} - f^*].$$
\end{proof}
The following iteration complexity result is an immediate consequence of Theorem~\ref{nonsmooth_convergence}.
\begin{corollary}
Let $\{x_k\}$ be the sequence generated by the SDP non-smooth algorithm. Given any $\epsilon >0,$ an iterate $x_k$ satisfying $f(x_k)-f^*\le\epsilon$ can be found in no more than
$$4M^2\mu^2 log_2\frac{f(x_0)}{\epsilon}$$
iterations, where $M$ is defined in \eqref{M_def}.
\end{corollary}
\begin{proof}
Follow Theorem~\ref{nonsmooth_convergence}, after each main Step, the objective function is decreased by one haft. That implies to obtain $\epsilon-$solution of SDP non-smooth formulation, we need $log_2\frac{f(x_0)}{\epsilon}$ restarts, then the number of iterations is 
$$4M^2\mu^2 log_2\frac{f(x_0)}{\epsilon}.$$ 
\end{proof}




\section{A smooth SDP Optimization Formulation for LMI}



In this Section, we introduce a smooth SDP optimization formulation for the Linear Matrix Inequality \eqref{SDP_inq}. Consider the following  
objective function
\beq \label{SDP_obj}
f(x)=\min_{u \in \mathcal{S}^n_-} \| \mathcal{A}x-B-u\|_F^2.
\eeq
Note that $f(x)$ is the square of the distance from $\mathcal{A}x-B$ to the non-positive semidefinite matrix cone $\mathcal{S}^n_-$. Our approach to solve the LMI problem \eqref{SDP_equivalent} is solving the equivalent optimization problem
$$\min_{x \in \bbr^n} f(x).$$
It is easy to see that if $x^*$ is a feasible solution to \eqref{SDP_equivalent} then $x^*$ is an optimal solution to $\min_{x \in \bbr^n} f(x)$ and vice versa. Furthermore, if $x^*$ is a feasible solution to \eqref{SDP_equivalent} then we also have $f(x^*)=0. $
The smoothness of the objective function $f(x)$ is presented in the following Lemma. 
\begin{lemma}
Given a linear operator $\mathcal{A}: \bbr^n \rightarrow \mathcal{S}^n,$ the objective function given in \eqref{SDP_obj} has $2\|\mathcal{A}\|^2$ Lipschitz continuous gradient, where $\| \mathcal{A}\|$ denotes the operator norm of $\mathcal{A}$ with respected to the pair of norm $\| .\|_2$ and $\| .\|_F$ defined as follows
\beq \label{Edef}
\|\mathcal{A}\|:=max\{\| \mathcal{A}u\|_F^*:\|u\|\le 1 \}
\eeq
\end{lemma}
\begin{proof}
The proof immediately follows the Proposition 1 of \cite{LaLuMo11-1}, in which
$$U=U^*=\bbr^n, V=V^*=\mathcal{S}^n,$$
and
$$\psi= (dist_{\mathcal{S}^n_-})^2, $$
where $dist_{\mathcal{S}^n_-}$ is the distance function to the cone $\mathcal{S}^n_-$ measured in terms of the norm $\| .\|_F$. Note that $dist_{\mathcal{S}^n_-}$ is a convex function with $2$-Lipschitz continuous gradient, see Proposition 15 of \cite{LaLuMo11-1}. 
\end{proof}
Define the Lipschitz constant of the objective function gradient by 
\beq \label{L_def}
L=2\|\mathcal{A}\|^2.
\eeq
It is easily to see that, the operator norm $\|\mathcal{A}\|$ can be computed as follows
$$\| \mathcal{A} \|= \| \mathcal{A} \|_{2,F}=\sqrt{\sum_{i=1}^n\| A_i\|_F^2}.$$
Throughout this paper, we will say that \eqref{SDP_obj} is a smooth optimization formulation of LMI problem. In next Subsections, we will describe our algorithms and discuss about their convergence behaviors. 

The smooth formulation can be solved by first order methods such as Nesterov's optimal method and its variants. In this Section, we propose a linearly convergent algorithm for solving the smooth formulation based on a global error bound for LMI. That error bound represents the growth condition of the objective function which is described in the following Lemma.
\begin{lemma}
For any $x \in \bbr^n$, we have
\beq \label{SDP_growth}
d^2(x,X^*) \le \mu^2f(x), 
\eeq
where $X^*$ is the feasible solution set of \eqref{SDP_equivalent}.
\end{lemma}
\begin{proof}
Note that $X^*$ is also the optimal set of minimizing \eqref{SDP_obj}. We consider the following two cases.

\noindent{\bf Case 1:} $x \in X^*,$ then
$$d(x,X^*)=0,$$
and
$$f(x)=\min_{u \in \mathcal{S}^n_-} \| \mathcal{A}x-B-u\|^2_F=0.$$
That implies \eqref{SDP_growth} is true for any $x \in X^*.$

\noindent{\bf Case 2:} $x \notin X^*,$ then by Corollary 1 in \cite{JoYe05}, we have
$$d(x,X^*) \le \mu \lambda_1(\mathcal{A}x-B).$$
Because $x \notin X^*$ then $\lambda_1(\mathcal{A}x-B) > 0.$ It is easy to show that
$$\lambda_1^2(\mathcal{A}x-B) \le \| \mathcal{A}x-B-u\|_F^2, \; \forall u \in \mathcal{S}^n_-.$$
That implies
$$d^2(x,X^*) \le \mu^2f(x), \; \forall x \notin X^*. $$
The Lemma immediately follows from two cases.
\end{proof}

Our algorithm is described as follows. At each main step (Step 1), our algorithm run the Nesterov's optimal method (see \cite{LaLuMo11-1}, \cite{Nest05-1}) for $K=4\mu \| \mathcal{A}\| $ iterations with the input is the current iterate to obtain the new one. In other words, we restart the Nesterov's algorithm after a constant number $K$ of iterations. We denote $\{x_k\}, k=0,1,...$ by the sequence obtained by our algorithm and $\{\bar x_i\}, i=0,1,...$ by the sequence obtained by Nesterov's method in the Step 1. The scheme of our algorithm is represented as follows.\\

{\bf The SDP Smooth Algorithm:}

{\bf Input:} $x_0^u \in \bbr^n$.

{\bf Output:} $x_k \in \bbr^n.$

1) $k^{th}$ iteration, $k \ge 1.$

\qquad \qquad Run Nesterov's algorithm with initial solution $\bar x_0=x_{k-1}$ for $K=4\mu \| \mathcal{A}\| $ iterations.

\qquad \qquad $x_k:=\bar x_K$.

2) Go to Step 1.\\








Observe that the above algorithm is different from running the Nesterov's algorithm for multiple times of $K$ iterations because when we restart the Nesterov's algorithm, the parameters is also restarted. The main convergence result is stated in the following theorem.
\begin{theorem} \label{convergenSDPsmooth}
The sequence $\{x_k\}, k=0,1,...$ generated by the SDP smooth algorithm satisfies
$$f(x_k) \le \frac{1}{2}f(x_{k-1}), \forall k \ge 1.$$
\end{theorem}
\begin{proof}
By convergence properties of Nesterov's algorithm (see \cite{LaLuMo11-1}, \cite{Nest05-1}), and note that $f(x)$ has $2\| \mathcal{A}\|^2$-Lipschitz continuous gradient, we have
$$f(x_k)-f^* \le \frac{8\| \mathcal{A}\|^2d^2(x_{k-1},X^*)}{K^2}.$$
Furthermore, $f^*=0,$ that implies
$$f(x_k) \le \frac{8\| \mathcal{A}\|^2d^2(x_{k-1},X^*)}{K^2}.$$
By Lemma~\ref{SDP_growth}, we have
$$d^2(x_{k-1},X^*) \le \mu^2 f(x_{k-1}).$$
Note that $K \ge 4\mu \| \mathcal{A},\|$ then
$$f(x_k) \le \frac{8\| \mathcal{A}\|^2\mu^2f(x_{k-1})}{K^2} \le \frac{1}{2}f(x_{k-1}).$$
\end{proof}
The following iteration complexity result is an immediate consequence of Theorem~\ref{convergenSDPsmooth}.
\begin{corollary}
Let $\{x_k\}$ be the sequence generated by the SDP smooth algorithm. Given any $\epsilon >0,$ an iterate $x_k$ satisfying $f(x_k)-f^*\le\epsilon$ can be found in no more than
$$4\mu \| \mathcal{A}\|log_2\frac{f(x_0)}{\epsilon}$$
iterations, where $\mathcal{A}$ is defined in \eqref{Edef}.
\end{corollary} 
\begin{proof}
Follow Theorem~\ref{convergenSDPsmooth}, after each main Step, the objective function is reduced by one haft. That implies to obtain $\epsilon-$solution of the SDP smooth formulation, we need $log_2\frac{f(x_0)}{\epsilon}$ restarts, then the number of iterations is 
$$4\mu \| \mathcal{A}\|log_2\frac{f(x_0)}{\epsilon}.$$ 
\end{proof}
%









\section{Uniformly linearly convergent algorithm for smooth and non-smooth formulations}
In the previous sections, we propose linearly convergent algorithms for non-smooth and smooth formulations in which both algorithms require estimating the Lipschitz constants, $M$ of the objective function and $L$ of its gradient. In this section, we present a new method which converges linearly not only for smooth formulation but also for non-smooth formulation. Moreover, this algorithm does not require any information of the problem such as the size of Lipschitz constants $L$ and $M$.

We consider a general convex programming problem of
\beq \label{CP}
f^*:=\min_{x \in \bbr^n} {f(x)},
\eeq
where the optimal value $f^*$ is known and $f(.)$ satisfies
\beq \label{smoothcondition}
f(y)-f(x)-\left \langle {f'(x),y-x}\right \rangle \le \frac{L}{2}\| y-x\|^2+M\| y-x\|, \;\; \forall x,y \in \bbr^n,
\eeq
for some $L,M \ge 0$ and $f'(x) \in \partial f(x).$ Clearly, this class of problems covers both non-smooth formulation \eqref{SDP_nonsmooth} corresponding to $L=0, M=\|\mathcal{A}\|$ and smooth formulation \eqref{SDP_obj} corresponding to $L=2\| \mathcal{A}\|^2, M=0,$ where the optimal values of both formulations are zeros. In \cite{Lan10-2}, Lan propose two algorithms which is uniformly optimal for solving both non-smooth and smooth convex programming problems. More interestingly, these algorithms do not require any smoothness information, such as the size of Lipschitz constants. We present in the next subsection a new algorithm, that can be viewed as a modification and combination of ABL and APL methods, posing uniformly linearly convergent rate for solving both smooth and non-smooth formulations and does not require any smoothness information of the problems as well. 

\subsection{The Modified ABL-APL algorithm}
The basic idea of bundle-level method is to construct a sequence of upper and lower bounds on $f^*$ whose gap converges to $0$. We introduce a gap reduction procedure which is much simple than those in ABL and APL method as follows.\\

{\bf The Modified ABL-APL gap reduction procedure:}

{\bf Input:} $x_0^u \in \bbr^n$.

{\bf Output:} $\bar x \in \bbr^n.$

{\bf Initialize:} Set $\bar f_0 = f(x_0^u)$ and $t=1$. Also let $x_0$ be arbitrary chosen, say $x_0=x^u_0$.

1) Set \beq \label{xtl_def} x_t^l=(1-\alpha_t)x^u_{t-1}+ \alpha_tx_{t-1}. \eeq

2) Update prox-center: Set
\beq \label{subprob}
x_t \in argmin\{ \|x-x_{t-1}\|^2:h(x_t^l,x) \le l\},
\eeq
where $l=f^*,$ and 
\beq \label{h_def} 
h(z,x):= f(z)+ \left \langle  {f'(z),x-z}\right \rangle.
\eeq

3) Update upper bound: Choose $x_t^u \in \bbr^n$ such that
$$f(x_t^u) \le \min \{ \bar f_{t-1},f(\alpha_tx_t+(1-\alpha_t)x^u_{t-1})\},$$
and set $\bar f_{t}=f(x_t^u)$. In particular, denoting $\tilde x^u_t \equiv \alpha_tx_t+(1-\alpha_t)x^u_{t-1},$ set $x^u_t=\tilde x^u_t$ if $f(\tilde x_t^u) \le \bar f_{t-1}$ and $x_t^u=x_{t-1}^u$ otherwise.

4) If $\bar f_t \le \frac{1}{2} \bar f_0,$ terminate the procedure with out put $\bar x=x^u_t.$

5) Set $t=t+1,$ and Go to Step 1.\\

This procedure is a modification and combination of ABL and APL gap reduction procedures. Firstly, in comparison with the ABL gap reduction procedure, in Step 1, we do not need to update the lower bound because the optimal value is known. Secondly, we use the same level $l=f^*$ for every step and each bundle contains only one cutting plane, that makes the difficulty of our subproblem is not increased after each iteration. Finally, the selection of the stepsizes $\alpha_t$ is different from the ABL gap reduction procedure, in particular, the selection is similar to the APL gap reduction procedure. \\

The algorithm is described as follows\\

{\bf The bundle-level method:}

{\bf Input:} Initial point $p_0 \in \bbr^n$ and tolerance $\epsilon >0.$

{\bf Initialize:} Set $ub_1=f(x_0)$ and $s=1$.

1) If $ub_s \le \epsilon,$ terminate;

2) Call the gap reduction procedure with input $x_0=p_s.$

Set $p_{s+1}=\bar x, ub_{s+1}=f(\bar x),$ where $\bar x$ is the output of the gap reduction procedure.

3) Set $s=s+1$ and Go to Step 1.\\

We say that {\it a phase of the Modified ABL-APL method} occurs whenever $s$ increments $1$ and an iteration performed by gap reduction procedure will be called {\it an iteration of the Modified ABL-APL method}. \\

According to the Modified ABL-APL gap reduction procedure, after each phase, the objective value is reduced by one half, corresponding to the case we set the constant factor in ABL or APL gap reduction procedure to $0.5$. To guarantee the linear convergence of our algorithm, we need to properly specify the stepsizes $\{\alpha_t\}$. Our stepsizes policies is similar to the APL method. More specifically, we denote
\beq \label{Gamma_def}
\Gamma _t : = \left\{ \begin{array}{l}
 1,\quad \quad \quad \quad \;\,t = 1 \\ 
 \Gamma _t (1 - \alpha _t ),\quad t \ge 2 \\ 
 \end{array} \right.,
\eeq
we assume that the stepsizes $\alpha_t \in (0,1], t \ge 1,$ are chosen such that
\beq \label{stepsize1}
\alpha _1  = 1,\quad \frac{{\alpha _t^2 }}{{\Gamma _t }} \le C_1 ,\quad \Gamma _t  \le \frac{{C_2 }}{{t^2 }}{\kern 1pt} \;\;\;\mbox{and}\;\;\;\Gamma _t \left[ {\sum\limits_{\tau  = 1}^t {\left( {\frac{{\alpha _\tau  }}{{\Gamma _\tau  }}} \right)^2 } } \right]^{\frac{1}{2}}  \le \frac{{C_3 }}{{\sqrt t }},\;\forall t \ge 1,
\eeq
\begin{lemma} \label{2stepsizes}
a) If $\alpha_t, t \ge 1,$ are set to
\beq \label{stepsizepolicy1}\alpha_t=\frac{2}{t+1},
\eeq
then the condition \eqref{stepsize1} holds with $C_1=2,C_2=2$ and $C_3=2/{\sqrt{3}};$

b) If $\alpha_t, t \ge 1,$ are computed recursively by
\beq \label{stepsizepolicy2}\alpha_1=\Gamma_1=1, \;\;\; \alpha_t^2=(1-\alpha_t) \Gamma_{t-1}=\Gamma_t, \; \forall t \ge 2,
\eeq
then we have $\alpha_t \in (0,1]$ for any $t \ge 2.$ Moreover, condition \eqref{stepsize1} holds with $C_1=1,C_2=4$ and $C_3=4/\sqrt{3};$
\end{lemma}
\begin{proof}
See Lemma 6 in \cite{Lan10-2}.
\end{proof}
It is worth noting that these stepsizes policies does not depend on any information of $L, M$ and $f(x_0).$ Furthermore, these stepsizes $\alpha_t$ is reset to $1$ at the start of a new phase, or in the other words, we reset the stepsizes whenever the objective value decreases by one half. \\

The main convergence properties of the above algorithm are described as follows.
\begin{theorem} \label{uniform_theorem}
Suppose that $\alpha_t \in (0,1], t \ge 1,$ in the Modified ABL-APL method are chosen such that \eqref{stepsize1} holds and $p_s, s \ge 0$ are generated by Modified ABL-APL method. Then, 

1) The total number iterations performed by the Modified ABL-APL method applied to the smooth formulation \eqref{SDP_obj} can be bounded by
\beq \label{uniform_smooth}
\left\lceil {\sqrt {LC_1 C_2 \mu } } \right\rceil \log _2 \frac{{f(p_0 )}}{\varepsilon },
\eeq
 where $L$ is given by \eqref{L_def};
 
 2) The total number iterations performed by the Modified ABL-APL method applied to the non-smooth formulation \eqref{SDP_nonsmooth} can be bounded by
\beq \label{uniform_nonsmooth}
\left\lceil {4M^2 C_3^2 \mu ^2 } \right\rceil \log _2 \frac{{f(p_0 )}}{\varepsilon },
\eeq
 where $M$ is given by \eqref{M_def};
 
\end{theorem}
Observe that, if the stepsizes polices \eqref{stepsizepolicy1} is chosen, then the total number of iterations performed by Modified ABL-APL applied to smooth and non-smooth formulations respectively are $$2 \sqrt{2} \| \mathcal{A}\|\mu \;\;\;\mbox{and}\;\;\; \frac{16}{3}M^2\mu^2log_2 \frac{f(x_0)}{\epsilon}.$$
On the other hand, if the stepsizes polices \eqref{stepsizepolicy2} is chosen, then the total number of iterations performed by Modified ABL-APL applied to smooth and non-smooth formulations respectively are $$2 \sqrt{2} \| \mathcal{A}\|\mu \;\;\;\mbox{and}\;\;\; \frac{64}{3}M^2\mu^2log_2 \frac{f(x_0)}{\epsilon}.$$ 

\subsection{Convergence analysis for Modified ABL-APL method}
In this section, we provide the proofs of our main results presented in the Theorem~\ref{uniform_theorem}. We first establish the convergence properties of the gap reduction procedure, which is the most important tool for proving Theorem~\ref{uniform_theorem}.

The following lemma shows that the reduction procedure generates a sequence of prox-centers $x_t$ which is "close" enough each other.
\begin{lemma} \label{distance_bound}
Suppose that $x_{\tau}, \tau =0,1,...,T,$ are the prox-centers generated by a reduction procedure, where $T$ is number of iterations performed, then we have
$$\sum_{\tau=1}^T \| x_{\tau}-x_{\tau-1}\|^2 \le \| x_0-x^*\|^2,$$
where $x^*$ is an arbitrary optimal solution of \eqref{CP}.
\end{lemma}
\begin{proof}
Denote the level sets by
$$\mathcal{L}_t:=\{ x \in \bbr^n: h(x_t^l,x) \le l\}, t=1,...,$$
and $x^*$ by an arbitrary optimal solution of \eqref{CP}. Then because of the convexity of the objective function, it is easy to see that $x^*$ is a feasible solution to \eqref{subprob} for every step, i.e. $x^* \in \mathcal{L}_t, t=1,2,...,T-1.$ Furthermore, using the Lemma 1 in \cite{Lan10-3} and the subproblem \eqref{subprob}, we have
$$\| x_{\tau}-x^*\|^2 + \| x_{\tau-1}-x_{\tau}\|^2 \le \| x_{\tau-1}-x^*\|^2, \tau =1,2,...,T.$$
Summing up the above inequalities we obtain
$$\| x_T-x^*\|^2 + \sum_{\tau=1}^T \| x_{\tau-1}-x_{\tau}\|^2 \le \| x_0-x^*\|^2, \forall x^* \in X^*.$$  
\end{proof}
The following result describes the main recursion for the Modified ABL-APL gap reduction procedure which together with the global error bound \eqref{SDP_growth} and \eqref{SDP_growth2} imply the rate of convergence of the Modified ABL-APL method.
\begin{lemma} \label{recursion}
Let $(x^l_t, x_t, x_t^u), t \ge 1,$ be the search points computed by the Modified ABL-APL gap reduction procedure. Also, let $\Gamma_t$ defined in \eqref{Gamma_def} and suppose that the stepsizes $\alpha_t, t \ge 1,$ are chosen such that the relation \eqref{stepsize1} holds. Then we have
\beq
f(x_t^u ) - f^* \le \frac{{LC_1 C_2 }}{{2t^2 }}\left\| {x_0  - x^* } \right\|^2  + \frac{{MC_3 }}{{\sqrt t }}\left\| {x_0  - x^* } \right\|
\eeq

\end{lemma} 
\begin{proof}
Denote $$\tilde x_t^u = \alpha_tx_t + (1- \alpha_t)x_{t-1}^u.$$ By definition of $x_t^l,$ we have
$$\tilde x_t^u-x_t^l=\alpha_t(x_t-x_{t-1}).$$
Using this observation, \eqref{smoothcondition}, \eqref{h_def}, \eqref{xtl_def}, \eqref{subprob} and the convexity of $f$, we have [Lemma 1]
\beq 
f(x_t^u) \le f(\tilde x_t^u) \le (1-\alpha_t)f(x^u_{t-1}) + \alpha_t l + \frac{L\alpha_t^2}{2} \| x_t-x_{t-1}\|^2 + M \alpha_t \| x_t-x_{t-1}\|.
\eeq
By subtracting $l$ from both sides of the above inequality, we obtain, for any $t \ge 1,$
\beq
f(x_t^u)-f^* \le (1-\alpha_t)[f(x_{t-1}^u)-f^* ]+ \frac{L\alpha_t^2}{2} \| x_t-x_{t-1}\|^2 + M \alpha_t \| x_t-x_{t-1}\|.
\eeq
Dividing both sides of above inequality to $\Gamma_t$ and using \eqref{Gamma_def}, \eqref{stepsize1}, we have
\begin{align*}
 \frac{{f(x_1^u ) - f^* }}{{\Gamma _1 }} & \le \frac{{1 - \alpha _1 }}{{\Gamma _1 }}\left[ {f(x_0^u ) - f^* } \right] + \frac{{LC_1 }}{2}\left\| {x_1  - x_0 } \right\|^2  + M\frac{{\alpha _1 }}{{\Gamma _1 }}\left\| {x_1  - x_0 } \right\| \cr  
  & = \frac{{LC_1 }}{2}\left\| {x_1  - x_0 } \right\|^2  + M\frac{{\alpha _1 }}{{\Gamma _1 }}\left\| {x_1  - x_0 } \right\|
\end{align*}
and for any $t \ge 2$
\begin{align*}
 \frac{1}{{\Gamma _t }}\left[ {f(x_t^u ) - f^* } \right] &\le \frac{{1 - \alpha _t }}{{\Gamma _t }}\left[ {f(x_{t - 1}^u ) - f^* } \right] + \frac{{LC_1 }}{2}\left\| {x_t  - x_{t - 1} } \right\|^2  + M\frac{{\alpha _t }}{{\Gamma _t }}\left\| {x_t  - x_{t - 1} } \right\| \cr
  &= \frac{1}{{\Gamma _{t - 1} }}\left[ {f(x_{t - 1}^u ) - f^* } \right] + \frac{{LC_1 }}{2}\left\| {x_t  - x_{t - 1} } \right\|^2  + M\frac{{\alpha _t }}{{\Gamma _t }}\left\| {x_t  - x_{t - 1} } \right\| 
 \end{align*}
Summing up the above inequalities, we have, for any $t \ge 1,$
\begin{align*}
 \frac{1}{{\Gamma _t }}\left[ {f(x_t^u ) - f^* } \right] &\le \frac{{LC_1 }}{2}\sum\limits_{\tau  = 1}^t {\left\| {x_\tau   - x_{\tau  - 1} } \right\|^2 }  + M\sum\limits_{\tau  = 1}^t {\frac{{\alpha _\tau  }}{{\Gamma _\tau  }}\left\| {x_\tau   - x_{\tau  - 1} } \right\|}  \cr
  &\le \frac{{LC_1 }}{2}\sum\limits_{\tau  = 1}^t {\left\| {x_\tau   - x_{\tau  - 1} } \right\|^2 }  + M\left[ {\sum\limits_{\tau  = 1}^t {\left( {\frac{{\alpha _\tau  }}{{\Gamma _\tau  }}} \right)^2 } } \right]^{\frac{1}{2}} \left[ {\sum\limits_{\tau  = 1}^t {\left\| {x_\tau   - x_{\tau  - 1} } \right\|^2 } } \right]^{\frac{1}{2}},  
 \end{align*}
where the second inequality follows from the Cauchy-Schwartz inequality. Then from the relation \eqref{stepsize1} and the Lemma~\ref{distance_bound}, for any $t \ge 1$ and $x^* \in X^*,$ we have
\begin{align*}
 f(x_t^u ) - f^*  &\le \frac{{LC_1 \Gamma _t }}{2}\left\| {x_0  - x^* } \right\|^2  + MC_3 \Gamma _t \left[ {\sum\limits_{\tau  = 1}^t {\left( {\frac{{\alpha _\tau  }}{{\Gamma _\tau  }}} \right)^2 } } \right]^{\frac{1}{2}} \left\| {x_0  - x^* } \right\| \cr
  &\le \frac{{LC_1 C_2 }}{{2t^2 }}\left\| {x_0  - x^* } \right\|^2  + \frac{{MC_3 }}{{\sqrt t }}\left\| {x_0  - x^* } \right\|  
 \end{align*}
\end{proof}
Now we are ready to prove the Theorem~\ref{uniform_theorem}.\\

\noindent{\bf Proof of Theorem~\ref{uniform_theorem}:} We will show that the MABL method obtain linear convergence rate for both smooth and non-smooth formulations. \\

First, we consider the smooth formulation, i.e. $M=0$. By the Lemma~\ref{recursion} and the error bound \eqref{SDP_growth}, note that $M=0,$ we have, for any $t \ge 1,$
\beq
f(x_t^u ) - f^*  \le \frac{{LC_1 C_2 }}{{2t^2 }}\left\| {x_0  - x^* } \right\|^2  \le \frac{{LC_1 C_2 \mu ^2 }}{{2t^2 }}\left[ {f(x_0^u ) - f^* } \right],
\eeq
that implies, for any $s \ge 1$
\beq
f(p_s)-f^* \le \frac{{LC_1 C_2 \mu ^2 }}{{2t^2 }}\left[ {f(p_{s-1} ) - f^* } \right].
\eeq
Then the number of iterations performed by reduction procedure each phase is bounded by $T_1,$ where
\beq
T_1 = \left\lceil {\sqrt {LC_1 C_2 \mu ^2 } } \right\rceil. 
\eeq
 After each phase, the objective function value is decreased by one half, then the number of phase is
 $$\max \{ 0, log_2 \frac{f(p_0)}{\epsilon}\}.$$
 Then Part 1 of Theorem~\ref{uniform_theorem} is automatically follows.\\
 
 Second, we consider the non-smooth formulation, i.e. $L=0$. By the Lemma~\ref{recursion} and the error bound \eqref{SDP_growth2}, note that $L=0,$ we have, for any $t \ge 1,$
\beq
f(x_t^u ) - f^*  \le \frac{{MC_3 }}{{\sqrt t }}\left\| {x_0  - x^* } \right\| \le \frac{{MC_3 \mu }}{{\sqrt t }}\left[ {f(x_0^u ) - f^* } \right],
\eeq
that implies, for any $s \ge 1,$
\beq 
f(p_s ) - f^*  \le \frac{{MC_3 \mu }}{{\sqrt t }}\left[ {f(p_{s-1} ) - f^* } \right].
\eeq
Then the number of iterations performed by reduction procedure each phase is bounded by $T_2,$ where
\beq
T_2 = \left\lceil {(2MC_3 \mu)^2   } \right\rceil. 
\eeq
 After each phase, the objective function value is decreased by one half, then the number of phase is
 $$\max \{ 0, log_2 \frac{f(p_0)}{\epsilon}\}.$$
 Then Part 2 of Theorem~\ref{uniform_theorem} is automatically follows.
 
\section{A special case}
In this section, we consider a linear system of inequalities, which is a special case of Linear Matrix Inequalities. 
 Interestingly, we still preserve a linearly convergent algorithm for solving linear inequalities system with a weaker assumption than Assumption~\ref{assumption1}. For convenience, in this section, we present a smooth formulation and the smooth algorithm for solving linear system of inequalities.

Consider the linear inequalities system
 \beq \label{linearsystem}
Ax \le b, \eeq
 or 
 $$
\left\{ \begin{array}{l}
 a_i^T x \le b_i \quad \left( {i \in I_ \le  } \right) \\ 
 a_i^T x = b_i \quad \left( {i \in I_ =  } \right) \\ 
 \end{array} \right.
$$
where $A$ is $m \times n$ matrix, $I_ \le, I_ =$ are index sets corresponding to inequalities and equalities. We assume that the following assumption holds.
 \begin{assumption}
 The feasible solution set of \eqref{linearsystem} is non empty.
 \end{assumption}
Note that this assumption is weaker than Assumption~\ref{assumption1} which requires the strict feasibility on the feasible solution set of LMIs. \\

We introduce the function $e: \bbr^m \rightarrow \bbr^m$ such that
$$
e(y)_i  = \left\{ \begin{array}{l}
 y_i^ +  \quad \left( {i \in I_ \le  } \right) \\ 
 y_i \quad \;\left( {i \in I_ =  } \right) \\ 
 \end{array} \right.,
$$
where
$$y_i^ +   = \max \left\{ {0,y_i } \right\}.
$$
Then, the equivalent optimization problem of linear inequalities system \eqref{linearsystem} is minimizing the objective function
\beq \label{smooth_ieq}f(x)=\frac{1}{2}\| e(Ax-b)\|^2, \eeq
Note that $x^*$ is a solution of \eqref{linearsystem} if and only if $x^*$ is optimal solution of \eqref{smooth_ieq}, and
$f(x^*)=0.$\\

The smoothness of this the objective function $f(x)$ is described in the following lemma.
\begin{lemma} \label{Lipschitz_gradient}
Given a matrix $A \in \bbr^{m \times n}$ and a vector $b \in \bbr^m$, the objective function given in \eqref{smooth_ieq} has a $\|A\|^2-$Lipschitz continuous gradient.
\end{lemma}
\begin{proof}
Denote $$C=\{ y \in \bbr^m: y_i \le 0 \;\mbox{for} \; i \in I_{\le}, y_i = 0 \;\mbox{for} \; i \in I_=\}.$$
Then $\| e(y)\|$ is the distance form a point $y$ to the closed, convex set $C$. Using Proposition 15 in \cite{LaLuMo11-1} it can be shown that $\| e(Ax-b)\|^2$ is differentiable with derivative given by 
$$\nabla f(x)=A^T(y-\Pi_C(y)), \; y \in \bbr^m$$
where $y=Ax-b$ and $\Pi_C(y)$ is projection of $y$ on $C.$

We have
\begin{align*}
&\| A^T(y_1 - \Pi_C (y_1))-A^T (y_2 - \Pi_C (y_2))\| \cr
&\le \| A\| \|[(y_1 - \Pi_C (y_1))-(y_2 - \Pi_C (y_2))]\| \cr
& \le \|A\| \| y_1 -y_2\| \cr
& = \| A\| \| A(x_1 -x_2)\| \cr
& \le \| A\|^2 \| x_1 -x_2\| \cr
\end{align*}
That implies 
$$\|\nabla f(x_1) - \nabla f(x_2) \| \le \| A\|^2 \| x_1-x_2\|.$$
\end{proof}

The growth condition of the objective function is described in the following lemma which was proposed by Hoffman, see \cite{Hoffman52}.
\begin{lemma} \label{errorbound2}
For any right-hand side vector $b \in \bbr^m$, let $S_b$ be the set of feasible solutions of the linear system \eqref{smooth_ieq}. Then there exists a constant $L_H$, independent of b, with the following property:
\beq \label{LS_growth}
x \in \bbr^m \; \mbox{and} \; S_b \ne \emptyset \Rightarrow d(x,S_b) \le L_H \| e(Ax-b)\|.
\eeq
\end{lemma} 
The objective function can be viewed as an error measure function which determines the errors in the corresponding equalities or inequalities of a given arbitrary point. This lemmas provide an error bound for the distance from a arbitrary point to the feasible solution set of \eqref{smooth_ieq}. The minimum constant $L_H$ satisfies the growth condition \eqref{LS_growth} is called the Hoffman constant which is well studied in \cite{Zhang99}, \cite{Pang97}, \cite{Li93}, \cite{GulHoffRoth95} and \cite{ZhengNg04}. That constant can be easily estimated in some cases, especially in linear system of equations. In that case, the Hoffman constant is the smallest non-zero singular value of the matrix A. 

The algorithm for linear system of inequalities is same as the smooth algorithm in Section 4. In particular, we restart the Nesterov accelerate gradient method after each $K=\sqrt{8\| A\|^2L_H^2}.$ The algorithm scheme is described as follows.

\begin{algorithm}

Step 0: Initial solution $x_0 \in \bbr^n.$

Step 1: $k^{th}$ iteration, $k \ge 1.$

\qquad \qquad Run Nesterov accelerate gradient method with initial solution $\bar x_0=x_{k-1}$ for $K=\sqrt{8\| A\|^2L_H^2}$ iterations.

\qquad \qquad $x_k:=\bar x_K$.

Step 2: Go to Step 1.

\end{algorithm}
The convergence property is described in following Theorem.
\begin{theorem} \label{convergence_LS}
For any $k \ge 1,$
$$f(x_k) \le \frac{1}{2}f(x_{k-1}).$$
\end{theorem}  
\begin{proof}
By convergence properties of Nesterov accelerate gradient method \cite{Nest83-1}, we have
$$f(x_k)-f^* \le \frac{4\| A\|^2}{(K+2)^2} \| x_{k-1}-x^* \|^2.$$
Using the Hoffman error bound and note that $f^*=0,$ we obtain
\begin{align*}
f(x_k) &\le \frac{4\| A\|^2L_H^2}{(K+2)^2}f(x_{k-1}) \cr
& \le \frac{1}{2}f(x_{k-1}),
\end{align*}
where the second inequality is followed by the choice of iterations number $K.$
\end{proof}
The following iterations complexity result is an immediate consequence of Theorem~\ref{convergence_LS}.
\begin{corollary}
Let $\{ x_k\}$ be the sequence generated by the smooth algorithm. Given any $\epsilon >0$, an iterate $x_k$ satisfying $f(x_k)-f^* \le \epsilon$ can be found in no more than
$$\sqrt{8\| A\|^2L_H^2}log_2\frac{f(x_0)}{\epsilon}$$
iterations.
\end{corollary}
\section{Conclusions}    
We present two formulations for a Linear Matrix Inequalities called smooth and non-smooth formulations. We propose two new first-order algorithms respectively, the SDP smooth algorithm and the SDP non-smooth algorithm, to solve these problem which can obtain a linear convergence rate. Basically, the idea of these algorithms is restarting an optimal method for smooth and non-smooth convex problem after a constant number of iterations and a global error bound for Linear Matrix Inequalities. These algorithms require knowledge about the smoothness parameters of the convex problems. We also introduce an uniformly linearly convergent algorithm for both formulations namely Modified ABL-APL method. This algorithm is a modification and combination of two algorithms, ABL and APL, proposed by Lan in \cite{Lan10-2}. Furthermore, no smoothness information of problem such as Lipschitz constant $L$ or $M$ is required. A special case of Linear Matrix Inequalities, Linear system of inequalities, is also considered in which we still obtain a linearly convergent algorithm under a weaker assumption than that in Linear Matrix Inequalities.  
\bibliographystyle{plain}
\bibliography{glan-bib}
\end{document}